\documentclass[a4paper]{article}
\usepackage[T1]{fontenc}
\usepackage[utf8]{inputenc}
\usepackage[babel]{csquotes}
\usepackage{microtype}
\usepackage{geometry} 
\usepackage{quoting} 
\quotingsetup{font=small}
\usepackage{amsmath}
\usepackage{amsfonts}
\usepackage{amssymb}
\usepackage{amsthm}
\usepackage{setspace}
\makeatletter
\renewcommand{\pod}[1]{\allowbreak\mathchoice
	{\if@display \mkern 18mu\else \mkern 8mu\fi (#1)}
	{\if@display \mkern 18mu\else \mkern 8mu\fi (#1)}
	{\mkern4mu(#1)}
	{\mkern4mu(#1)}
}
\usepackage{mathrsfs}
\usepackage{mathtools}
\usepackage{braket}
\usepackage{bm}
\usepackage{enumerate}
\usepackage{empheq}
\usepackage{emptypage}
\usepackage{graphics}
\usepackage{tikz}
\usetikzlibrary{matrix,arrows,decorations.pathmorphing}
\usepackage{tikz-cd}
\usepackage{xparse}
\usepackage{stmaryrd}
\usepackage{latexsym,xspace,centernot}
\usepackage{pgf}
\usepackage{authblk}
\usepackage{hyperref}
\hypersetup{colorlinks=true,
	linkcolor=blue,
	anchorcolor=blak,
	citecolor=red,menucolor=blak,linktocpage=true}
\newcommand{\Z}{\mathbb{Z}}
\newcommand{\C}{\mathbb{C}}

\newcommand{\R}{\mathbb{R}}
\newcommand{\N}{\mathbb{N}}
\renewcommand{\epsilon}{\varepsilon}
\newcommand{\m}[4]{\begin{pmatrix}
		#1&#2\\#3&#4
	\end{pmatrix}}
	\renewcommand{\epsilon}{\varepsilon}
	\newcommand{\g}[1]{\mathfrak{#1}}
	\newcommand{\Sel}{\mathcal{S}}
	
	\theoremstyle{plain}
	
	\newtheorem{theorem}{Theorem}

	\newtheorem{corollary}{Corollary}
	\theoremstyle{remark}
	\newtheorem{remark}{Remark}
	
	\theoremstyle{definition}

	\numberwithin{equation}{section}

	\DeclareMathOperator{\res}{Res}
	
\begin{document}
		\title{\textbf{On the linear twist of degree 1 functions in the extended Selberg class}}
		\author{Giamila Zaghloul}
		\affil{Dipartimento di Matematica\\Università degli Studi di Genova\\via Dodecaneso 35, 16146 Genova}
		\date{}
		\maketitle
		\begin{abstract}
			Given a degree 1 function $F\in\Sel^{\sharp}$ and a real number $\alpha$, we consider the linear twist $F(s,\alpha)$, proving that it satisfies a functional equation reflecting $s$ into $1-s$, which can be seen as a Hurwitz-Lerch type of functional equation. We also derive some results on the distribution of the zeros of the linear twist. 
		\end{abstract}
		\section{Introduction}\label{section_selberg}
		In 1989, Selberg \cite{sel} presented his axiomatic definition of the class of $L$-functions. For a complete overview of the theory we refer to the surveys of Kaczorowski~\cite{k} and Perelli \cite{p1}, \cite{p2}. In this work we focus on the so called \emph{extended Selberg class} $\Sel^{\sharp}$, which is the class of the non identically zero Dirichlet series 
		\[
		F(s)=\sum_{n=1}^{\infty}\frac{a(n)}{n^s}	\]
		absolutely convergent for $\Re(s)=\sigma>1$, admitting a meromorphic continuation to the complex plane and satisfying a functional equation of the form \[
		\Phi(s)=\omega \overline{\Phi}(1-s),\quad\text{where}\quad \Phi(s)=Q^s\prod_{j=1}^{r}\Gamma(\lambda_js+\mu_j)F(s)=\gamma(s)F(s),
		\]
		with $Q>0$, $\lambda_j>0$, $\Re(\mu_j)\geq 0$, $|\omega|=1$ and $\overline{\Phi}(s)=\overline{\Phi(\overline{s})}$ (see e.g. \cite{p1} for the precise definition).
		Given $F\in \Sel^\sharp$, the \emph{degree} of $F$ is defined by
		\[
		d=2\sum_{j=1}^{r}\lambda_j.
		\]
		The degree is an \emph{invariant} for $F$, i.e. it is uniquely determined by the function itself and does not depend on the shape of the functional equation, which is not unique (cf. e.g. \cite[p. 27]{p1} or \cite{k}).\\ 
		Then, $\Sel^{\sharp}$ splits into the disjoint union of the subclasses $\Sel^{\sharp}_d$ of the functions with given degree $d\geq 0$. The so-called \emph{Strong degree conjecture} states that $\Sel^{\sharp}_d=\emptyset$ if $d\notin\N$.  
		So far, the conjecture is proved in the range $0< d<2$. In particular, several authors independently proved that $\Sel^{\sharp}_d=\emptyset$ for $0<d<1$, (cf. \cite{ric}, \cite{boc}, \cite{co-gh}, \cite{mol}), while the proof for $1<d<2$ is due to Kaczorowski and Perelli \cite{k-p_vii}. For degrees $d=0$ and $d=1$, the elements of $\Sel$ and $\Sel^{\sharp}$ have been completely characterized. The results below describe the structure of $\Sel^{\sharp}_0$ and $\Sel^{\sharp}_1$, \cite[Theorems 1 and 2]{k-p_i}.
		\begin{theorem}[Kaczorowski-Perelli]\phantomsection\label{Sel_0}
			\begin{itemize}
				\item[(i)] Let $F\in\Sel^{\sharp}_0$. Then $q\in\N$, the pair $(q,\omega)$ is an invariant for $F$ and $\Sel^{\sharp}_0$ is the disjoint union of the subclasses $S^{\sharp}_0(q,\omega)$, with $q\in\N$ and $|\omega|=1$.
				\item[(ii)] Let $F\in\Sel^{\sharp}_0(q,\omega)$, with $q,\omega$ as above. Then $F(s)$ is a Dirichlet polynomial of the form 
				\begin{equation}\label{D-poly}
				F(s)=\sum_{n\mid q}\frac{a(n)}{n^s}.
				\end{equation}
				\item[(iii)] $V^{\sharp}_0(q,\omega)=\Sel^{\sharp}_0(q,\omega)\cup \{0\}$ is a real vector space of dimension $d(q)=\sum_{d\mid q}1$. 
			\end{itemize}
			
		\end{theorem}
		
		\begin{remark}
			As shown in \cite[\S 3]{k-p_i} the functional equation in the case $d=0$ implies the following relation on the coefficients
			\begin{equation}\label{coeff_0}
			a(n)=\frac{\omega}{\sqrt{q}}n\overline{a}\bigg(\frac{q}{n}\bigg)\quad\text{for}\quad n\mid q.
			\end{equation}
		\end{remark}
		Let now $\chi$ be a Dirichlet character modulo $q$. Denote by $\chi^*$ the primitive character inducing $\chi$ and by $f_{\chi}$ its conductor. If $\tau_{\chi^*}$ is the Gauss sum corresponding to $\chi^*$, let 
		\[
		\omega_{\chi^*}=\frac{\tau_{\chi^*}}{i^\g{a}\sqrt{f_{\chi}}}\quad\text{where}\quad\g{a}=\begin{cases}
		0\quad\text{if}\quad \chi(-1)=1\\
		1\quad\text{if}\quad \chi(-1)=-1.
		\end{cases}
		\] 
		If $d=1$ the \emph{root-number} is defined as 
		\[
		\omega^*=\omega(\beta Q^2)^{i\theta}\prod_{j=1}^{r}\lambda^{-2i\Im(\mu_j)}_j,\quad\text{where}\quad \beta=\prod_{j=1}^{r}\lambda^{2\lambda_j}_j.
		\]
		Moreover, write
		\[
		\g{X}(q,\xi)=\begin{cases}
		\Set{\chi\pmod q| \chi(-1)=1}\quad\text{if}\quad \eta=-1\\
		\Set{\chi\pmod q|\chi(-1)=-1}\quad\text{if}\quad \eta=0.
		\end{cases}
		\]
		With the above notation, a complete characterization of $\Sel^{\sharp}_1$ is given.
		\begin{theorem}[Kaczorowski-Perelli]\phantomsection\label{teor_S1}
			\begin{itemize}
				\item[(i)] Let $F\in \Sel^{\sharp}_1$. Then $q\in\N$, $\eta\in\{-1,0\}$ and the triple $(q,\xi,\omega^*)$ is an invariant. Moreover, $\Sel^{\sharp}_1$ is the disjoint union of the subclasses $\Sel^{\sharp}_1(q,\xi,\omega^*)$, with $q\in \N$, $\eta\in\{-1,0\}$, $\theta\in\R$ and $|\omega^*|=1$. 
				\item[(ii)] Let $F\in\Sel^{\sharp}_1(q,\xi,\omega^*)$, with $q$, $\xi$, $\omega^*$ as above. Then, $F(s)$ can be uniquely written as 
				\begin{equation}\label{structure_Sel_1}
				F(s)=\sum_{\chi\in \mathfrak{X}(q,\xi)}P_{\chi}(s+i\theta)L(s+i\theta,\chi^*),
				\end{equation}
				where $P_{\chi}$ is a Dirichlet polynomial in $\Sel^\sharp_0(q/f_{\chi},\omega^*\overline{\omega}_{\chi^*})$ and $L(s,\chi^*)$ the Dirichlet $L$-function associated to the primitive character $\chi^*$. 
				\item[(iii)] If $a(n)$ is the $n$-th Dirichlet coefficient of $F\in\Sel^{\sharp}_1$, then $\widetilde{a}(n)=a(n)n^{i\theta}$ is periodic of period $q$.
				\item[(iv)] $V^{\sharp}_1(q,\xi,\omega^*)=\Sel^{\sharp}_1(q,\xi,\omega^*)\cup\{0\}$ is a real vector space of dimension
				\[
				\dim V^{\sharp}_1(q,\xi,\omega^*)=\begin{cases}
				\big\lfloor\frac{q}{2}\big\rfloor\quad&\text{if}\quad \xi=-1\\ \big\lfloor\frac{q-1-\eta}{2}\big\rfloor&\text{otherwise}.
				\end{cases}
				\] 
			\end{itemize}
		\end{theorem}
		\begin{remark}\label{oss_a=eta+1}
			Observe that, since all the characters in $\g{X}(q,\xi)$ have the same parity, $\g{a}$ is completely determined by $\eta$, and hence by $\xi$ which is an invariant for $F\in\Sel^{\sharp}_1$. In particular we have $\g{a}=\eta+1$. 
		\end{remark}
		The main tool in the proof of Theorem \ref{teor_S1} is the so-called \emph{linear twist}, defined for $\sigma>1$ by 
		\begin{equation}\label{linear_twist}
		F(s,\alpha)=\sum_{n=1}^{\infty}\frac{a(n)}{n^s}e(-n\alpha),
		\end{equation}
		where $F\in\Sel^{\sharp}_1$, $\alpha\in\R$ and $e(x):=e^{2\pi ix}$. In \cite[Theorem 7.1]{k-p_i}, Kaczorowski and Perelli established some analytic properties of the linear twists, such as the meromorphic continuation to the half-plane $\sigma>0$ and the possible existence of a simple pole at $s=1-i\theta$. 
				
		In \cite{k-p_iv}, Kaczorowski and Perelli showed that the theorem for the linear twists is a special case of a general result holding for $\Sel^{\sharp}_d$ for any degree $d>0$. Given $\alpha\in\R$ and $F\in\Sel^{\sharp}_d$, with $d>0$, for $\sigma>1$ they introduced the so-called \emph{standard twist}
		\begin{equation}\label{standard_twist}
		F_d(s,\alpha)=\sum_{n=1}^{\infty}\frac{a(n)}{n^s}e(-n^{1/d}\alpha),
		\end{equation}
		which corresponds to the linear twist when $d=1$. In \cite[Theorem 1]{k-p_iv}, the main analytic properties of the standard twists are described.
		
		 In \cite{grado2}, Kaczorowski and Perelli focused on degree 2 functions. They considered the standard twists of a Hecke $L$-function associated to a cusp form of half-integral weight, deriving a functional equation which can be seen as a degree 2 analogue of the Hurwitz-Lerch functional equation. In this case, the functional equation is obtained thanks to the special form of the involved $\Gamma$-factor, which enables the explicit computation of a certain hypergeometric function arising in the proof.
		 In this work we will focus on the case $d=1$. We are interested in investigating further analytic properties of the linear twists. In particular, as a first step we derive a functional equation. Then, we go on studying the growth on vertical strips and the distribution of the zeros. We remark that, thanks to the characterization given in Theorem \ref{teor_S1}, the linear twists of degree 1 functions in $\Sel^{\sharp}$ are closely related to Hurwitz-Lerch zeta functions.
		
		\section{A functional equation for the linear twist}
		Let $F\in\Sel^{\sharp}_1$ and $\alpha\in\R$. Since $F(s,\alpha)=F(s,\{\alpha\})$, we can assume $\alpha\in(0,1]$. For $\beta\in\R$, let
		\begin{equation}\label{F*}
		F_*(s,\alpha,\beta):=\sum_{n+\beta>0}\frac{\widetilde{a}(n)}{(n+\beta)^{s+i\theta}}e(-n\alpha),
		\end{equation}
		where $widetilde{a}(n)$ is as in Theorem \ref{teor_S1}. 
		\begin{remark}
			In the definition above $n\in\Z$ and $\widetilde{a}(n)$ is extend to $\Z$ by periodicity. 
		\end{remark}
		It can be easily seen that equation \eqref{F*} with $\beta=0$ coincides with the linear twist $F(s,\alpha)$. In the above notation, our goal is now proving the following result. 
		\begin{theorem}\label{th_f.e.}
			Let $F\in\Sel^{\sharp}_1$ and let $\alpha\in(0,1]$. Then, the linear twist $F(s,\alpha)$ satisfies the functional equation
			\begin{equation}\label{f.e.}
			F(1-s,\alpha)=\frac{\omega^*\Gamma(s-i\theta)q^{s-i\theta-\frac{1}{2}}}{i^{\g{a}}(2\pi)^{s-i\theta}}\bigg(e^{i\frac{\pi}{2}(s-i\theta)}\overline{F}_*(s,0,-\alpha q)+(-1)^{\g{a}}e^{-i\frac{\pi}{2}(s-i\theta)}\overline{F}_*(s,0,\alpha q)\bigg).
			\end{equation}
		\end{theorem}
		\begin{remark}
			Observe that the functional equation \eqref{f.e.} is not of Riemann type, even if it still reflects $s$ into $1-s$. As explained below, it can be seen as a Hurwitz-Lerch type functional equation. 
		\end{remark}
		The key point to derive \eqref{f.e.} is Theorem \ref{teor_S1}, in particular expression \eqref{structure_Sel_1}. So, assume that for any $\chi\in\g{X}(q,\xi)$, $P_{\chi}$ is a Dirichlet polynomial with coefficients $c(n)=c_{\chi}(n)$ for $n\mid \frac{q}{f_{\chi}}$. We rewrite the linear twist as 
		\[
		\begin{split}
		F(s,\alpha)&=\sum_{\chi\in \mathfrak{X}(q,\xi)}\sum_{\substack{n\mid q/f_{\chi}\\m\geq 1}}\frac{c(n)\chi^*(m)}{(mn)^{s+i\theta}}e(-mn\alpha)
		=\sum_{\chi\in \mathfrak{X}(q,\xi)}\sum_{n\mid q/f_{\chi}}\frac{c(n)}{n^{s+i\theta}}L(s+i\theta,\chi^*,n\alpha),
		\end{split}
		\]
		where $L(s+i\theta,\chi^*,n\alpha)=\sum_{m\geq 1}\frac{\chi^*(m)}{m^{s+i\theta}}e(-mn\alpha)$ is the linear twist of the Dirichlet $L$-function associated to the primitive character $\chi^*$.
		\subsection{A functional equation for $L(s,\chi,\alpha)$}
		As a first step, we derive a functional equation for the linear twist of a Dirichlet $L$-function associated to a primitive character. So, let $\chi$ be a primitive Dirichlet character modulo $q$ and let $\alpha\in\R$. Using the orthogonality properties of characters we get
		\begin{equation}\label{L_hurwitz_combination}
		\begin{split}
		L(s,\chi,\alpha)&=\sum_{n=1}^{\infty}\frac{\chi(n)}{n^s}e(-n\alpha)=\sum_{n=1}^{\infty}\frac{\chi(n)}{n^s}e(-n\{\alpha\})
		=\frac{1}{\tau_{\overline{\chi}}}\sum_{a=0}^{q-1}\overline{\chi}(a)\zeta_L(s,a/q-\{\alpha\},0),
		\end{split}\end{equation}
		where $\zeta_L(s,x,y)$ is the \emph{Hurwitz-Lerch zeta function} defined as
		\begin{equation}\label{def_h-l}
		\zeta_L(s,x,y)=\sum_{n>-\{y\}}\frac{e(n\{x\})}{(n+\{y\})^s},
		\end{equation}
		for $x,y\in\R$ and $n\in\Z$. It is well known that the Hurwitz-Lerch zeta function can be analytically continued to a holomorphic function in $\C$ with a possible simple pole of residue 1 at $s=1$ if and only if $x\in\Z$. Moreover, it satisfies a functional equation of the form
		\begin{equation}\label{h-l}
		\zeta_L(1-s,x,y)=\frac{\Gamma(s)}{(2\pi)^s}\bigg(e^{i\frac{\pi}{2}s-2\pi i\{x\}\{y\}}\zeta_L(s,-y,x)+e^{-i\frac{\pi}{2}s+2\pi i\{-x\}\{y\}}\zeta_L(s,y,-x)\bigg).
		\end{equation}
		We refer e.g. to Garunkstis-Laurincikas \cite{g-l} for a detailed discussion on the properties of the Hurwitz-Lerch zeta function.
		Now, for $\alpha,\beta\in\R$, we introduce the notation 
		\[
		L_*(s,\chi,\alpha,\beta)=\sum_{n+\beta>0}\frac{\chi(n)}{(n+\beta)^s}e(-n\alpha),
		\]
		observing that $L_*(s,\chi,\alpha,0)=L(s,\chi,\alpha)$. Then, the following result holds.
		\begin{theorem}\label{th_L-function}
			Let $L(s,\chi)$ be the Dirichlet $L$-function associated to the  primitive character $\chi$ modulo $q$.  Then, given $\alpha\in\R$, the linear twist $L(s,\chi,\alpha)$ admits a meromorphic continuation to $\C$ with a possible simple pole at $s=1$. Moreover, it satisfies the functional equation 
			\begin{equation}\label{f.e.L_function}
			L_*(1-s,\chi,\alpha,0)=\frac{\Gamma(s)\tau_{\chi}\chi(-1)}{(2\pi)^sq^{1-s}}\bigg(e^{i\frac{\pi}{2}s}L_*(s,\overline{\chi},0,-\alpha q)+\chi(-1)e^{-i\frac{\pi}{2}s}L_*(s,\overline{\chi},0,\alpha q)\bigg).
			\end{equation}
		\end{theorem}
		\begin{proof}
			Writing $L(s,\chi,\alpha)$ as a linear combination of Hurwitz-Lerch zeta function as in \eqref{L_hurwitz_combination}, we deduce that the linear twist can be extended to a meromorphic function on $\C$ with a possible simple pole at $s=1$. Given $a\in\{0,\dots,q-1\}$ with $(a,q)=1$, the pole at $s=1$ of $\zeta_L(s,a/q-\{\alpha\},0)$ exists if and only if $\frac{a}{q}-\{\alpha\}\in\Z$. Then, $L(s,\chi,\alpha)$ has a pole at $s=1$ if and only if $\chi(\alpha q)\ne 0$ (we assume that $\chi(x)=0$ if $x\notin\Z$). The residue is $\res_{s=1}L(s,\chi,\alpha)=\frac{\overline{\chi}(\alpha q)}{\tau_{\overline{\chi}}}$.\\
			On the other hand, by \eqref{L_hurwitz_combination} and \eqref{h-l}, we get 
			\begin{equation}\label{f.e.L}	\begin{split}
			&L(1-s,\chi,\alpha)=\frac{1}{\tau_{\overline{\chi}}}\sum_{a=0}^{q-1}\overline{\chi}(a)\zeta_L(1-s,a/q-\{\alpha\},0)\\&
			=\frac{1}{\tau_{\overline{\chi}}}\sum_{a=0}^{q-1}\overline{\chi}(a)\frac{\Gamma(s)}{(2\pi)^s}\bigg(e^{i\frac{\pi}{2}s}\zeta_L(s,0,a/q-\{\alpha\})+e^{-i\frac{\pi}{2}s}\zeta_L(s,0,-a/q+\{\alpha\})\bigg)
			\end{split}
			\end{equation}
			Let now
			\[
			\Sigma_1=\sum_{a=0}^{q-1}\overline{\chi}(a)\zeta_L(s,0,a/q-\{\alpha\})
			\quad\text{and}\quad\Sigma_2=\sum_{a=0}^{q-1}\overline{\chi}(a)\zeta_L(s,0,-a/q+\{\alpha\}).
			\]
			Observing that $\frac{a}{q}-\{\alpha\}\in (-1,1)$ and properly rearranging the sums, we get
			\[
			\Sigma_1=\sum_{\substack{m>\alpha q\\ \{\frac{m}{q}\}>\{\alpha\}}}\frac{\overline{\chi}(m)}{(m-\alpha q)^s}+\sum_{\substack{m> \alpha q\\\{\frac{m}{q}\}\leq\{\alpha\}}}\frac{\overline{\chi}(m)}{(m-\alpha q)^s}=\sum_{m> \alpha q}\frac{\overline{\chi}(m)}{(m-\alpha q)^s}=L_*(s,\overline
			\chi,0,-\alpha q),
			\]
			and similarly
			\[
			\Sigma_2=\sum_{\substack{m>-\alpha q\\\{-\frac{m}{q}\}\geq\{\alpha\}}}\frac{\overline{\chi}(m)}{(m+\alpha q)^s}+\sum_{\substack{m>-\alpha q\\\{-\frac{m}{q}\}<\{\alpha\}}}\frac{\overline{\chi}(m)}{(m+\alpha  q)^s}=\sum_{m>-\alpha q}\frac{\overline{\chi}(m)}{(m+\alpha  q)^s}=L_*(s,\overline{\chi},0,\alpha q). 
			\]
			Then, recalling that $\frac{\tau_{\chi}\tau_{\overline{\chi}}}{q}=\chi(-1)$, equation \eqref{f.e.L} can be rewritten as 
			\[
			L_*(1-s,\chi,\alpha,0)=\frac{\Gamma(s)}{(2\pi)^s}\frac{\tau_{\chi}\chi(-1)}{q^{1-s}}\bigg(e^{i\frac{\pi}{2}s}L_*(s,\overline{\chi},0,-\alpha q)+\chi(-1)e^{-i\frac{\pi}{2}s}L_*(s,\overline{\chi},0,\alpha q)\bigg).
			\]\end{proof}
		\begin{remark}
			It can be noticed that \eqref{f.e.L_function} has a shape which is similar to \eqref{h-l}. For this reason, we say that \eqref{f.e.L_function} is a Hurwitz-Lerch type of functional equation. The same holds for \eqref{f.e.}.
		\end{remark}
		\subsection{A functional equation for the linear twist of $P_\chi(s)L(s,\chi^*)$}
		Let now $F\in\Sel^{\sharp}_1$ and write  \begin{equation}\label{sumF_chi}
		F(s)=\sum_{\chi\in\mathfrak{X}(q,\xi)}F_{\chi}(s+i\theta),\quad\text{where}\quad F_{\chi}(s):=P_{\chi}(s)L(s,\chi^*).
		\end{equation} Given $\chi\in\g{X}(q,\xi)$, assume that $P_{\chi}(s)=\sum_{n\mid q/f_{\chi}}\frac{c(n)}{n^s}$. The linear twist of $F_{\chi}(s)$ becomes 
		\[
		\begin{split}
		F_{\chi}(s,\alpha)&:=\sum_{\substack{n\mid q/f_{\chi}\\m\geq 1}}\frac{c(n)\chi^*(m)}{(mn)^s}e(-mn\alpha)
		=\sum_{n\mid q/f_{\chi}}\frac{c(n)}{n^{s}}L(s,\chi^*,n\alpha).
		\end{split}
		\]
		We recall that the following relation holds
		\begin{equation}\label{poly}
		\frac{c(n)}{n}=\frac{\omega^*{\overline{\omega}}_{\chi^*}}{\sqrt{q/f_{\chi}}}\overline{c}\bigg(\frac{q}{nf_{\chi}}\bigg),
		\end{equation}
		coming from the functional equation for $\Sel^{\sharp}_0$ (cf. \eqref{coeff_0}). Combining \eqref{poly} with \eqref{f.e.L_function}, we get
		\[
		\begin{split}
		&F_{\chi}(1-s,\alpha)=\sum_{n\mid q/f_{\chi}}\frac{c(n)}{n^{1-s}}L_*(1-s,\chi^*,n\alpha,0)\\&=
		\sum_{n\mid q/f_{\chi}}\frac{c(n)}{n^{1-s}}\bigg(\frac{\Gamma(s)}{(2\pi)^s}\frac{\tau_{{\chi}^*}\chi(-1)}{f_{\chi}^{1-s}}\bigg(e^{i\frac{\pi}{2}s}L_*(s,\overline{\chi}^*,0,-\alpha nf_{\chi})+\chi(-1)e^{-i\frac{\pi}{2}s}L_*(s,\overline{\chi}^*,0,\alpha nf_{\chi})\bigg)\bigg)
		\\&=\frac{\omega^*f_{\chi}^s}{i^{\mathfrak{a}}\sqrt{q}}\frac{\Gamma(s)}{(2\pi)^s}
		\bigg(e^{i\frac{\pi}{2}s}\sum_{\substack{n\mid q/f_{\chi}\\m/n-\alpha f_{\chi}>0}}\frac{\overline{c}(q/nf_{\chi})\overline{\chi}^*(m)}{(m/n-\alpha f_{\chi})^s}+\chi(-1)e^{-i\frac{\pi}{2}s}\sum_{\substack{n\mid q/f_{\chi}\\m/n+\alpha f_{\chi}>0}}\frac{\overline{c}(q/nf_{\chi})\overline{\chi}^*(m)}{(m/n+\alpha f_{\chi})^s}\bigg)\\&=\frac{\omega^*q^{s-\frac{1}{2}}\Gamma(s)}{i^{\mathfrak{a}}(2\pi)^s}
		\bigg(e^{i\frac{\pi}{2}s}\sum_{\substack{n\mid q/f_{\chi}\\mn-\alpha q>0}}\frac{\overline{c}(n)\overline{\chi}^*(m)}{(mn-\alpha q)^s}+\chi(-1)e^{-i\frac{\pi}{2}s}\sum_{\substack{n\mid q/f_{\chi}\\mn+\alpha q>0}}\frac{\overline{c}(n)\overline{\chi}^*(m)}{(mn+\alpha q)^s}\bigg),
		\end{split}
		\]where we rearranged the sum over $n$, observing that if  $n\mid \frac{q}{f_{\chi}}$, then $n'=\frac{q}{nf_{\chi}}$ also divides $\frac{q}{f_{\chi}}$. Moreover, we used the following identities \[\overline{\omega}_{\chi^*}=\omega_{\overline{\chi}^*}=\frac{\tau_{\overline{\chi}^*}}{i^{\mathfrak{a}}\sqrt{f_{\chi}}}\quad\text{and}\quad\frac{\tau_{\chi^*}\tau_{\overline{\chi}^*}}{f_{\chi}}\chi(-1)=1.\]
		For $\alpha\in(0,1]$ and $\beta\in\R$, let
		\[
		F_{\chi}(s,\alpha,\beta):=\sum_{\substack{n\mid q/f_{\chi}\\mn+\beta>0}}\frac{c(n)\chi^*(m)}{(mn+\beta)^s}e(-mn\alpha).
		\]
		It can be easily observed that
		\[
		\sum_{\substack{n\mid q/f_{\chi}\\mn\pm\alpha q>0}}\frac{\overline{c}(n)\overline{\chi}^*(m)}{(mn\pm\alpha q)^s}=\overline{F}_{\chi}(s,0,\pm\alpha q).
		\]
		Then, for the linear twist of $F_{\chi}(s)$ we derive the functional equation 
		\begin{equation}\label{f.e.F_chi}
		F_{\chi}(1-s,\alpha)=\frac{\omega^*q^{s-\frac{1}{2}}\Gamma(s)}{i^{\mathfrak{a}}(2\pi)^s}(e^{i\frac{\pi}{2}s}\overline
		{F}_{\chi}(s,0,-\alpha q)+\chi(-1)e^{-i\frac{\pi}{2}s}\overline{F}_{\chi}(s,0,\alpha q)).
		\end{equation}
		\subsection{Proof of Theorem \ref{th_f.e.}}
		Let now $F\in\Sel^{\sharp}_1$.
		The functional equation for $F(s,\alpha)$ comes from the results of the previous sections. By \eqref{sumF_chi}, we have
		\[
		\begin{split}
		&F(1-s,\alpha)=\sum_{\chi\in \mathfrak{X}(q,\xi)}F_{\chi}(1-s+i\theta,\alpha)\\&=\sum_{\chi}\bigg(\frac{\omega^*\Gamma(s-i\theta)q^{s-i\theta-\frac{1}{2}}}{i^{\g{a}}(2\pi)^{s-i\theta}}\big(e^{i\frac{\pi}{2}{s-i\theta}}\overline{F}_{\chi}(s-i\theta,0,-\alpha q)+\chi(-1)e^{-i\frac{\pi}{2}{s-i\theta}}\overline{F}_{\chi}(s-i\theta,0,\alpha q)\big)\bigg)\\&
		=\frac{\omega^*\Gamma(s-i\theta)q^{s-i\theta-\frac{1}{2}}}{i^{\g{a}}(2\pi)^{s-i\theta}}\bigg(e^{i\frac{\pi}{2}(s-i\theta)}\overline{F}_*(s,0,-\alpha q)+(-1)^{\g{a}}e^{-i\frac{\pi}{2}(s-i\theta)}\overline{F}_*(s,0,\alpha q)\bigg),
		\end{split}
		\]
		since we have 
		\[
		\sum_{\chi\in\g{X}(q,\xi)}\overline{F}_{\chi}(s-i\theta,0,\pm \alpha q)=\sum_{\chi\in \mathfrak{X}(q,\xi)}\sum_{\substack{n\mid q/f_{\chi}\\mn+\pm \alpha q>0}}\frac{\overline{c}(n)\overline{\chi}^*(m)}{(mn\pm\alpha q)^{s-i\theta}}=\overline{F}_*(s,0,\pm \alpha q). 
		\]
		\begin{remark}
			As already observed, it is well-known that the linear twist $F(s,\alpha)$ has a meromorphic continuation to $\C$ with a possible simple pole at $s=1-i\theta$. We now briefly sketch the calculation of the residue. Recall that
			\[
			F(s,\alpha)=\sum_{\chi\in \mathfrak{X}(q,\xi)}\sum_{n\mid q/f_{\chi}}\frac{c(n)}{n^{s+i\theta}}L(s+i\theta,\chi^*,n\alpha).
			\]
			Since $\res_{s=1-i\theta}L(s+i\theta,\chi^*,\alpha n)=\frac{\overline{\chi}^*(\alpha nf_{\chi})}{\tau_{\overline{\chi}^*}}$, using again \eqref{poly} and writing $m=\frac{q}{nf_{\chi}}$, we get 
			\[
			\begin{split}
			\res_{s=1-i\theta}F(s,\alpha)&=\sum_{\chi\in \mathfrak{X}(q,\xi)}\sum_{n\mid q/f_{\chi}}\frac{c(n)}{n}\frac{\overline{\chi}^*(\alpha nf_{\chi})}{\tau_{\overline{\chi}^*}}\\&
			=\frac{\omega^*}{i^{\g{a}}\sqrt{q}}\sum_{\chi\in\g{X}(q,\xi)}\sum_{m\mid q/f_{\chi}}\overline{c}(m)\overline{\chi}^*\bigg(\frac{\alpha q}{m}\bigg)=\frac{\omega^*}{i^{\g{a}}\sqrt{q}}\overline{\widetilde{a}(\alpha q)}. 
			\end{split}
			\]
			Note that, as stated in \cite[Theorem 1]{k-p_iv}, the pole exists if and only if $a(\alpha q)\ne 0$. 
		\end{remark}
		\section{The order of growth}
		The functional equation allows us to go on studying the analytic properties of the linear twist. We start investigating the order of growth on vertical strips. It is already known by \cite[Theorem 2]{k-p_iv} that the linear twist has polynomial growth on vertical strips. However, we consider the Lindel\"{o}f function associated to $F(s,\alpha)$,
		\begin{equation}\label{lindelof}
		\mu(\sigma,\alpha)=\inf\set{\xi\in\R| F(\sigma+it,\alpha)\ll |t|^{\xi}\text{ as }|t|\to +\infty}. 
		\end{equation}
		Let now 
		\[
		\mu^*(\sigma,\alpha):=\max(\mu^+(\sigma,\alpha),\mu^-(\sigma,\alpha)), \] where \[\mu^{\pm}(\sigma,\alpha):=\inf\set{\xi\in\R | \overline{F}_*(\sigma+it,0,\mp\alpha q)\ll |t|^{\xi}\text{ as } t\to \pm\infty}.
		\]
		In the above setting, one can deduce the following corollary. 
		\begin{corollary}\label{th_poly_growth}
			Let $F\in\Sel^{\sharp}_1$ and $\alpha\in(0,1]$. Then , the linear twist $F(s,\alpha)$ has polynomial growth on vertical strips and the corresponding Lindel\"{o}f function satisfies 
			\begin{equation}\label{lind}
			\mu(\sigma,\alpha)=\frac{1}{2}-\sigma+\mu^*(1-\sigma,\alpha).
			\end{equation}
		\end{corollary}
		\begin{proof}
			The result easily follows with standard methods, combining the functional equation \eqref{f.e.} with Stirling's formula for the $\Gamma$-factor. 
		\end{proof}
		\begin{remark}
			Since the linear twist $F(s,\alpha)$ is absolutely convergent for $\sigma>1$ and the Lindel\"{o}f function is continuous, we get
			\[
			\mu(\sigma,\alpha)=\begin{cases}
			0 \quad&\text{if}\quad \sigma\geq 1\\
			\frac{1}{2}-\sigma& \text{if}\quad \sigma\leq 0.
			\end{cases}
			\]
			Moreover, by the convexity of the Lindel\"of function we deduce the upper bound
			\[
			\mu(\sigma,\alpha)\leq \frac{1-\sigma}{2} \quad\text{if}\quad 0<\sigma<1. 
			\]
		\end{remark}
		\section{Distribution of the zeros}
		We are now interested in studying the distribution of the zeros of the linear twist. Theorem \ref{zeros_line} and \ref{zeros_rat} below concern the zeros outside the critical strip $0<\sigma<1$, while Corollary \ref{cor_r-v-m} is an analogue of the Riemann-von Mangoldt formula. As a first step, we consider the zeros in the left half-plane $\sigma<0$ coming from the interaction between the two terms on the right-hand side of the functional equation. We refer to these zeros as the \emph{trivial zeros} of $F(s,\alpha)$. Since the $\Gamma$-factor does not vanish, the zeros of the linear twist are the zeros of the function
		\begin{equation}\label{defH}
		H(s):=e^{i\frac{\pi}{2}(1-s-i\theta)}\overline{F}_*(1-s,0,-\alpha q)+(-1)^{\g{a}}e^{-i\frac{\pi}{2}(1-s-i\theta)}\overline{F}_*(1-s,0,\alpha q).
		\end{equation}
		The theorem below shows that, for $\sigma$ sufficiently small, the linear twist has infinitely many zeros located inside circles whose centers lie on certain generalized arithmetic progressions. 		
		\begin{theorem}\label{zeros_line}
			There exist infinitely many circles $C_h$, $h\geq 0$, of center $s_h=\alpha h+\beta$, with $\alpha,\beta\in\C$, $\Re(\alpha),\Re(\beta)<0$, and radius $\eta^{-\Re(s_h)}$ for some $0<\eta<1$, such that $F(s,\alpha)$ has exactly one zero inside each circle. 
		\end{theorem}		
		\begin{proof}
			Recalling that the coefficients $\tilde{a}(n)$ are defined in (iii) of Theorem 2, let
			\begin{equation}\label{m1m2}
			m_1=\min\Set{m>\alpha q|\widetilde{a}(m)\ne 0}\quad\text{and}\quad m_2=\min\Set{m>-\alpha q|\widetilde{a}(m)\ne 0},
			\end{equation} then
			\[
			\overline{F}_*(1-s,0,-\alpha q)
			=\frac{\overline{\widetilde{a}(m_1)}}{(m_1-\alpha q)^{1-s-i\theta}}+\sum_{m>m_1}\frac{\overline{\widetilde{a}(m)}}{(m-\alpha q)^{1-s-i\theta}}
			\]
			and similarly
			\[
			\overline{F}_*(1-s,0,\alpha q)
			=\frac{\overline{\widetilde{a}(m_2)}}{(m_2+\alpha q)^{1-s-i\theta}}+\sum_{m>m_2}\frac{\overline{\widetilde{a}(m)}}{(m+\alpha q)^{1-s-i\theta}}.
				\]
				
			Moreover, recalling \eqref{defH}, we write $H(s)=W(s)+V(s)$, where
			\begin{equation}\label{def_W}
			W(s):=e^{i\frac{\pi}{2}(1-s-i\theta)}\frac{\overline{\widetilde{a}(m_1)}}{(m_1-\alpha q)^{1-s-i\theta}}+(-1)^{\mathfrak{a}}e^{-i\frac{\pi}{2}(1-s-i\theta)}\frac{\overline{\widetilde{a}(m_2)}}{(m_2+\alpha q)^{1-s-i\theta}}.
			\end{equation}
			and
			\begin{equation}\label{def_V}
			V(s)=e^{i\frac{\pi}{2}(1-s-i\theta)}\sum_{m>m_1}\frac{\overline{\widetilde{a}(m)}}{(m-\alpha q)^{1-s-i\theta}}+(-1)^{\mathfrak{a}}e^{-i\frac{\pi}{2}(1-s-i\theta)}\sum_{m>m_2}\frac{\overline{\widetilde{a}(m)}}{(m+\alpha q)^{1-s-i\theta}}.
			\end{equation}
			In this scenario, the idea is to study the zeros of $W(s)$ and then to apply Rouché's theorem to localize those of $H(s)$. Let $\overline{\widetilde{a}(m_1)}=\rho_1e^{i\theta_1}$ and $\overline{\widetilde
				{a}(m_2)}=\rho_2e^{i\theta_2}$, with $\rho_1,\rho_2>0$ and $\theta_1,\theta_2\in [0,2\pi)$. Then $W(s)=0$ if and only if
			\[
			\begin{split}
			&e^{i\frac{\pi}{2}(1-\sigma)+\frac{\pi}{2}(t+\theta)}e^{\log\rho_1+i\theta_1}e^{(\sigma-1+i(t+\theta))\log (m_1-\alpha q)}\\&=e^{\pi(\mathfrak{a}+1)}e^{-i\frac{\pi}{2}(1-\sigma)-\frac{\pi}{2}(t+\theta)}e^{\log\rho_2+i\theta_2}e^{(\sigma-1+i(t+\theta))\log (m_2+\alpha q)}.
			\end{split}
			\]
			The equality of the moduli of the two sides gives
			\begin{equation}\label{line}
			\ell=\ell(\alpha): t+\frac{1}{\pi}\sigma\log\bigg(\frac{m_1-\alpha q }{m_2+\alpha q}\bigg)+\frac{1}{\pi}\log\bigg(\frac{\rho_1(m_2+\alpha q)}{\rho_2(m_1-\alpha q)}\bigg)+\theta=0,
			\end{equation}
			while from the arguments we get, for $k\in\Z$,
			\begin{equation}
			\ell_k=\ell_k(\alpha): t\log\bigg(\frac{m_1-\alpha q}{m_2+\alpha q}\bigg)-\pi\sigma+\theta_1-\theta_2+\theta\log\bigg(\frac{m_1-\alpha q}{m_2+\alpha q}\bigg)+(2k+\mathfrak{a})\pi=0.
			\end{equation}
			Observe that the above lines are orthogonal. Then, as $k$ runs over the integers, $W(s)$ has infinitely many zeros in the half-plane $\sigma<0$ lying on the line $\ell$. We denote these zeros as $s_k=\sigma_k+it_k\in\ell$, $k\in\Z$, observing that they form a generalized arithmetic progression.\\
			Let now $s=\sigma+it\in \ell$ and $\delta>0$. Define $t^*=t+\delta$ and $s^*=\sigma+it^*$. For $\delta$ sufficiently small,
			\begin{equation}\label{w}
			\begin{split}
			W(s^*)&\gg
			e^{\frac{\pi}{2}\delta}e^{\frac{\pi}{2}(t+\theta)}\frac{\rho_1}{(m_1-\alpha q)^{1-\sigma}}-e^{-\frac{\pi}{2}\delta}e^{-\frac{\pi}{2}(t+\theta)}\frac{\rho_2}{(m_2+\alpha q)^{1-\sigma}}\\&=((m_1-\alpha q)(m_2+\alpha q))^\frac{\sigma-1}{2}(\rho_1\rho_2)^{\frac{1}{2}}(e^{\frac{\pi}{2}\delta}-e^{-\frac{\pi}{2}\delta})\\&\gg ((m_1-\alpha q)(m_2+\alpha q))^\frac{\sigma-1}{2}(\rho_1\rho_2)^{\frac{1}{2}}\delta.
			\end{split}
			\end{equation}
			To derive an upper bound for $|V(s^*)|$, we denote
			\[\widetilde{m}_1=\min\Set{m>m_1|\widetilde{a}(m)\ne 0}\quad\text{and}\quad\widetilde{m}_2=\min\Set{m>m_2|\widetilde{a}(m)\ne 0}.\]
			Using the integral criterion and observing that the sums over $\chi$ and $n$ are finite and the set $\Set{\frac{\mid\overline{c}(n)\mid}{n}|\chi\in\g{X}(q,\xi),n|\frac{q}{f_{\chi}}}$ is bounded, we get respectively
			\[
			\sum_{m\geq \widetilde{m}_1}\frac{\overline{|\widetilde{a}(m)}|}{(m-\alpha q)^{1-\sigma}}
			\ll\frac{1}{(\widetilde{m}_1-\alpha q)^{-\sigma}}\quad\text{and}\quad
			\sum_{m\geq \widetilde{m}_2}\frac{|\overline{\widetilde{a}(m)}|}{(m+\alpha q)^{1-\sigma}}\ll  \frac{1}{(\widetilde{m}_2+\alpha q)^{-\sigma}}.
			\]
			It follows that 
			\begin{equation}\label{v}
			\begin{split}
			V(s^*)&\ll e^{\frac{\pi}{2}(t^*+\theta)}\frac{1}{(\widetilde{m}_1-\alpha q)^{-\sigma}}+e^{-\frac{\pi}{2}(t^*+\theta)}\frac{1}{(\widetilde{m}_2+\alpha q)^{-\sigma}}\\&
			\ll 
			((m_1-\alpha q)(m_2+\alpha q))^{\frac{\sigma-1}{2}}\bigg(A\bigg(\frac{m_1-\alpha q}{\widetilde{m}_1-\alpha q}\bigg)^{1-\sigma}+B\bigg(\frac{m_2+\alpha q}{\widetilde{m}_2+\alpha q}\bigg)^{1-\sigma}\bigg),
			\end{split}
			\end{equation}
			where $A=\big(\frac{\rho_2}{\rho_1}\big)^{\frac{1}{2}}(\widetilde{m}_1-\alpha q)$ and $B=\big(\frac{\rho_1}{\rho_2}\big)^{\frac{1}{2}}(\widetilde{m}_2+\alpha q)$.\\
			We now observe that $\frac{m_1-\alpha q}{\widetilde{m}_1-\alpha q}<1$, so $\big(\frac{m_1-\alpha q}{\widetilde{m}_1-\alpha q}\big)^{1-\sigma}\to 0$ as $\sigma\to -\infty$ and similarly the other term. Then, if $\delta=\eta_0^{-\sigma}$, with 
			\begin{equation}\label{eta0}
			\max\bigg(\frac{m_1-\alpha q}{\widetilde{m}_1-\alpha q},\frac{m_2+\alpha q}{\widetilde{m}_2+\alpha q}\bigg)<\eta_0<1,
			\end{equation}
			combining equations \eqref{w} and \eqref{v}, for $\sigma$ sufficiently small we have $|W(s^*)|-|V(s^*)|>0$. The same result follows by the same argument with $s^*=\sigma+i(t-\delta)$, $\delta>0$. \\ Since $s=\sigma+it$ varies on the line $\ell$, we have proved that for $\sigma>-\sigma'$, with a suitable $\sigma'>0$, $|W(s)|-|V(s)|>0$
			on the boundary of the region
			\[
			\g{L}_{\delta}=\g{L}_{\delta}(\alpha)=\Set{s\in \C \text{ with distance}<\delta \text{ from the line }\ell }.
			\] 
			Let now $k\in\Z$ and consider the line $l_k$.
			Given $\delta>0$ sufficiently small, assume that $s\in\ell_k\pm\delta$. Then, we can prove
			\begin{equation}\label{W}
			\begin{split}
			W(s)&
			\gg\delta \max\big(e^{\frac{\pi}{2}(t+\theta)}\rho_1(m_1-\alpha q)^{\sigma-1},e^{-\frac{\pi}{2}(t+\theta)}\rho_2(m_2+\alpha q)^{\sigma-1}\big).
			\end{split}\end{equation}
			Moreover, we already know that 
			\[
			V(s)\ll \max\bigg(e^{\frac{\pi}{2}(t+\theta)}(m_1-\alpha q)^{\sigma-1}\bigg(\frac{m_1-\alpha q}{\widetilde{m}_1-\alpha q}\bigg)^{1-\sigma},e^{-\frac{\pi}{2}(t+\theta)}(m_2+\alpha q)^{\sigma-1}\bigg(\frac{m_2+\alpha q}{\widetilde{m}_2+\alpha q}\bigg)^{1-\sigma} \bigg).
			\]
			\\
			So gathering equation \eqref{W} and the above upper bound, we can state that there exist $\sigma''\geq 0$ such that for $\sigma<-\sigma''$, $|W(s)|-|V(s)|>0$
			when $s\in \ell_k\pm \delta$ for any sufficiently small $\delta>0$.\\
			Let now \begin{equation}\label{sigmabar}
			\overline{\sigma}=\max(\sigma',\sigma'').
			\end{equation} Applying Rouché's theorem, we can state that there exists $\eta\in(\eta_0,1)$ with $\eta_0$ as in \eqref{eta0} such that for each zero $s_k=\sigma_k+it_k$ of $W(s)$ with $\sigma_k<-\overline{\sigma}$, $F(s,\alpha)$ has exactly one zero in a circle with center in $s_k$ and radius $\eta^{-\sigma_k}$. Re-parameterizing the zeros the statement follows.
		\end{proof}
		\begin{remark}
			The result above corresponds to Corollary 2 in \cite{grado2}, but actually our theorem is slightly more precise. Indeed, Kaczorowski and Perelli showed that the zeros are located around a certain line, and observed that the problem of the finer definition of the trivial zeros is open. On the other hand, we prove that the zeros in our case are located inside circles
			with centers lying in a generalized arithmetic progression on $\ell$ and radius tending to zero as $\sigma\to-\infty$.
		\end{remark}
		We now present another theorem on the distribution of the zeros. In this case, the proof is complete only if $\alpha$ is rational, while for irrational values of $\alpha$ only partial results are known. 
		\begin{theorem}\label{zeros_rat}
			Let $0<\alpha\leq 1$ be rational. If $F(s,\alpha)$ and $\overline{F}_*(s,0,\pm\alpha q)$ are not of the form $P(s)L(s,\chi)$, where $P(s)$ is a Dirichlet polynomial and $L(s,\chi)$ is a Dirichlet $L$-function, then
			\begin{itemize}
				\item[(i)] there exist $\sigma_1,\sigma'_1>0$ such that the set 
				\[
				\{\sigma\in(1,1+\sigma_1]\mid F(\sigma+it,\alpha)=0,\text{ for some }t\in\R\}
				\]
				is dense in $(1,1+\sigma_1]$ and $F(\sigma+it,\alpha)\ne 0$ if $\sigma>1+\sigma'_1$. 
				\item[(ii)] there exist $\sigma_2>0$ such that the set 
				\[
				\{\sigma\in[-\sigma_2,0)\mid F(\sigma+it,\alpha)=0,\text{ for some }t\in\R\}
				\]		
				is dense in the interval $[-\sigma_2,0)$. 
			\end{itemize}
		\end{theorem}
		\begin{proof}
			If $\alpha$ is rational, $F(s,\alpha)$, and $\overline{F}(s,0,\pm \alpha q)$  can be written as linear combinations of Dirichlet $L$-functions with Dirichlet polynomials as coefficients, since they have periodic coefficients (cf. \cite[Theorem PDCB]{s-w}). Then, by the result of Saias and Weingartner \cite[Theorem]{s-w}, if these linear combinations do not reduce to a single term, they have infinitely many zeros in the half-plane $\sigma>1$. 
			The density of the real parts and the possible existence of gaps in the region where the zeros exist are a consequence of \cite[Theorem 1.1]{r2}. 
			Therefore, part $(i)$ is proved. 
			
			Let now $\sigma<0$. The main tool in the proof of assertion $(ii)$ is the functional equation. Consider again the function $H(s)$ in \eqref{defH}. We write
			\begin{equation}\label{H}
			\begin{split}
			H(s)&=e^{i\frac{\pi}{2}(1-s-i\theta)}\overline{F}_*(1-s,0,-\alpha q)+(-1)^{\mathfrak{a}}e^{-i\frac{\pi}{2}(1-s-i\theta)}\overline{F}_*(1-s,0,\alpha q)\\&=e^{i\frac{\pi}{2}(1-s-i\theta)}F_1(s)+(-1)^{\mathfrak{a}}e^{-i\frac{\pi}{2}(1-s-i\theta)}F_2(s).
			\end{split}    
			\end{equation}
			Observe that, if they are not of the form $P(s)L(s,\chi)$,  $F_1(s)$ and $F_2(s)$ have infinitely many zeros, since $1-\sigma>1$. Moreover, the exponential factors imply that if one of the two terms of $H(s)$ tends to infinity, the other tends to zero.
			
			Assume $t=\Im(s)>0$. Let $\rho$ be a zero of $F_1(s)$ and consider $\delta>0$ sufficiently small such that $F_1$ does not vanish on a circle of center $\rho$ and radius $\delta$. Define 
			\begin{equation}\label{gamma}
			\gamma=\underset{|s|=\delta}{\min}|F_1(s+\rho)|>0.
			\end{equation}
			Since we are working with generalized Dirichlet series, by almost periodicity we can assert that, for any $\varepsilon>0$, the set of $\tau\in\R$ such that 
			\begin{equation}\label{almost_periodic}
			\underset{|s|=\delta}{\max}|F_1(s+\rho+i\tau)-F_1(s+\rho)|<\varepsilon
			\end{equation}
			is relatively dense (i.e. there exists $\ell>0$ such that any interval of length $\ell$ contains at least a $\tau$ as above). Moreover, we have 
			\[
			|e^{-\frac{\pi}{2}(1-s-i\theta)}H(s)-F_1(s)|=|e^{-\pi(1-s-i\theta)}F_2(s)|=e^{-\pi(t+\theta)}|F_2(s)|.
			\]
			Then, the polynomial growth on vertical strips (Corollary \ref{th_poly_growth}) implies, for some positive $A$,
			\begin{equation}\label{poly_gr}
			\underset{|s|=\delta}{\max}|H(s+\rho+i\tau)e^{-i\frac{\pi}{2}(1-s-\rho-i\tau)}-F_1(s+\rho+i\tau)|\ll e^{-\pi\tau}\tau^A.
			\end{equation}
			We gather equations \eqref{gamma}, \eqref{almost_periodic} and \eqref{poly_gr}, choosing $\varepsilon$ and $\tau$ such that $\varepsilon+e^{-\pi\tau}\tau^A<\gamma$. Then, by triangular inequality we get
			\[\underset{|s|=\delta}{\max}|H(s+\rho+i\tau)e^{-i\frac{\pi}{2}(1-s-\rho-i\tau)}-F_1(s+\rho)|<	\gamma=\underset{|s|=\delta}{\min}|F_1(s+\rho)|.\]
			Applying Rouché's theorem, we deduce that $F_1(s)$ and $H(s)$ have the same number of zeros inside the circle of center $\rho$ and radius $\delta$. Since $F_1(\rho)=0$, we conclude that $H(s)$, then $F(s,\alpha)$, has a zero inside the considered circle.\\	
			The same argument applies for $t<0$, replacing $F_1(s)$ with $F_2(s)$, since in this case the second term is dominating in $H(s)$. This concludes the proof of part $(ii)$. 
		\end{proof}
		If $\alpha$ is irrational, $\overline{F}_*(s,0,\pm\alpha q)$ can be seen as generalized Hurwitz zeta functions with periodic coefficients. Thus, by our result on the zeros of generalized Hurwitz zeta functions
		\cite{zag2},
		we deduce that they have infinitely many zeros for $\sigma>1$.
		Therefore, if $\sigma<0$, $\overline{F}_*(1-s,0,\pm\alpha q)$ have infinitely many zeros (cf. \eqref{H}) and the argument used in the proof of $(ii)$ applies. Thus, part $(ii)$ of Theorem \ref{zeros_rat} even holds for $\alpha$ irrational. On the other hand, the existence of infinitely many zeros $F(s,\alpha)$ in the right half-plane, is still an open problem if $\alpha$ is not rational, since the analogue for the classical Hurwitz-Lerch zeta function is still not known.\\   	
		
		We now want to derive an analogue of the Riemann-von Mangoldt formula for the linear twist. Let $\overline{\sigma}$ be as in \eqref{sigmabar} and let $\sigma'_1$ be as in Theorem \ref{zeros_rat} (then $1+\sigma'_1$ is an upper bound of the real parts of the zeros). We define as \emph{non-trivial zeros} the zeros in the strip $-\overline{\sigma}\leq \sigma\leq 1+\sigma'_1$. Let
		\[
		N_F(T,\alpha)=\sharp\Set{\rho=\beta+i\gamma| F(\rho,\alpha)=0, -\overline{\sigma}\leq\beta\leq 1+\sigma '_1, |\gamma|\leq T}
		\]
		be the counting function of the non-trivial zeros and let $\overline{n}$ be the smallest integer $n$ such that $a(n)\ne 0$. Then, recalling \eqref{m1m2}, the following result holds.  
		\begin{corollary}\label{cor_r-v-m}
			Let $\alpha\in(0,1]$ and $F\in \Sel^{\sharp}_1$. Then, as $T\to \infty$, 
			\begin{equation}\label{r-v-m}
			N_F(T,\alpha)=\frac{T}{\pi}\log T+\frac{T}{\pi}\log\bigg(\frac{q}{2\pi e \overline{n}\sqrt{(m_1-\alpha q)(m_2+\alpha q)}}\bigg)+ O(\log T).
			\end{equation}
		\end{corollary}
		\begin{proof}
			Given $T_0>0$, $a>\overline{\sigma}$, $b>1+\sigma'_1$ sufficiently large, define $N^{\pm}(T)$ respectively as
			\[
			N^+(T)=\sharp\Set{\rho=\beta+i\gamma| F(\rho,\alpha)=0, -a\leq\beta\leq b, T_0<\gamma\leq T},
			\]
			and
			\[
			N^-(T)=\sharp\Set{\rho=\beta+i\gamma| F(\rho,\alpha)=0, -a\leq\beta\leq b, -T\leq\gamma< -T_0}.
			\]
			Then, \[
			N_F(T,\alpha)=N^+(T)+N^-(T)+O(1).
			\]
			The result follows by a suitable application of the argument principle to $N^{\pm}(T)$. We consider the rectangle joining the points $-a+iT_0$, $b+iT_0$, $b+iT$ and $-a+iT$ to calculate $N^+(T)$. Similarly, for $N^-(T)$ we apply the same argument to the rectangle in the lower half-plane joining the points $-a-iT$, $b-iT$, $b-iT_0$ and $-a-iT_0$. The proof proceeds exactly as in \cite[Corollary 3]{grado2}, the differences being in the coefficient of $T \log T$ in \eqref{r-v-m}, halved since here we are in degree 1, as well as in the coefficient of $T$ in \eqref{r-v-m}, where again the change of degree is visible.
		\end{proof}
		\begin{remark}
			We conclude observing that, if $\alpha$ is rational and $\alpha\notin\{\frac{1}{2},1\}$, the linear twist $F(s,\alpha)$ has infinitely many zeros in the strip $\frac{1}{2}<\sigma<1$ and the real parts of these zeros are dense in the interval $(\frac{1}{2},1)$. This results is an immediate consequence of Theorem 2 of \cite{k-k}, since in this case \[
			F(s,\alpha)=\sum_{j=1}^{N}P_j(s)L(s,\chi_j)\quad\text{with}\quad N\geq 2.
			\]
		
		\end{remark}
		
	\end{document}